\documentclass[12pt,a4paper]{article}
\usepackage[english]{babel}

\usepackage{pstricks}
\usepackage{pst-node}
\usepackage{amsmath,boxedminipage}
\usepackage{amssymb}
\usepackage{graphicx}
\usepackage{enumerate}
\usepackage{color}
\usepackage[text={15cm,24cm}]{geometry}
\usepackage{float}

\newenvironment{proof}{{\bf Proof}:\ }%
{~\ \hfill $\Box$\vspace{0,5cm}}

{~\ \hfill$\Diamond$\vspace{0,5cm}}

\newtheorem{prop}{Property}[section]
\newtheorem{theorem}{Theorem}[section]

\newtheorem{lemma}[theorem]{Lemma}

\newtheorem{coro}[theorem]{Corollary}

\newrgbcolor{lightlightlightgray}{0.9 0.9 0.9}

\numberwithin{equation}{section}

\begin{document}

\title{On the complexity of Dominating Set for graphs with fixed diameter}
\author{
    Valentin Bouquet\footnotemark[2]
    \and
    Fran\c cois Delbot\footnotemark[3]
    \and
    Christophe\ Picouleau\footnotemark[1] \footnotemark[2]
    \and
    St\'ephane Rovedakis\footnotemark[2]
}

\date{\today}

\def\thefootnote{\fnsymbol{footnote}}

\footnotetext[1]{ \noindent
Corresponding author: {\tt chp@cnam.fr}}

\footnotetext[2]{ \noindent
Conservatoire National des Arts et M\'etiers, CEDRIC laboratory, Paris (France). Email: {\tt
valentin.bouquet@lecnam.net,chp@cnam.fr,stephane.rovedakis@cnam.fr}
}

\footnotetext[3]{ \noindent
Sorbonne Universit\'e, Laboratoire d'Informatique de Paris 6 (LIP6), Paris (France). Email: {\tt
francois.delbot@lip6.fr
}}

\graphicspath{{.}{graphics/}}

\maketitle
\begin{abstract}
    A set $S\subseteq V$ of a graph $G=(V,E)$ is a dominating set if each vertex has a neighbor in $S$ or belongs to $S$. \textsc{Dominating Set} is the problem of deciding, given a graph $G$ and an integer $k\geq 1$, if $G$ has a dominating set of size at most $k$. It is well known that this problem is $\mathsf{NP}$-complete even for claw-free graphs. We give a complexity dichotomy for \textsc{Dominating Set} for the class of claw-free graphs with diameter $d$. We show that the problem is $\mathsf{NP}$-complete for every fixed $d\ge 3$ and polynomial time solvable for $d\le 2$. To prove the case $d=2$, we show that \textsc{Minimum Maximal Matching} can be solved in polynomial time for $2K_2$-free graphs.

    \vspace{0.2cm}
\end{abstract}

\section{Introduction}
The graphs we are concerned are finite, undirected, simple and loopless.  The reader is referred to \cite{Bondy} and  \cite{GJ} for, respectively,  the definitions and notations on graph theory  and on computational complexity.

Given a graph $G=(V,E)$, a set $S\subseteq V$ is called a {\it dominating set} if for every $v\in V\setminus S$, $N(v)\cap S\neq \emptyset$. For a dominating set $S$ of $G$, we say that $S$ {\it dominates} $G$. The minimum cardinality of a dominating set in $G$ is denoted by $\gamma(G)$ and called the \textit{domination number}. A dominating set $S$ with $\vert S\vert=\gamma(G)$ is called a {\it minimum dominating set}. Following \cite{DomBook}, such a set is called a $\gamma$-set of $G$. \medskip

The decision problem associated with the minimum dominating set is defined as:

\begin{center}
    \begin{boxedminipage}{.99\textwidth}
        \textsc{\sc Dominating Set} \\[2pt]
        \begin{tabular}{ r p{0.8\textwidth}}
            \textit{~~~~Instance:} & a graph $G=(V,E)$ and an integer  $k\geq 1$. \\
            \textit{Question:}     & is $\gamma(G)\leq k$ ?
        \end{tabular}
    \end{boxedminipage}
\end{center}

Our aim is to determine the computational complexity of computing a $\gamma$-set or the domination number for graphs of fixed diameter. The paper is organized as follows. In the following section we give some definitions and notations. Section \ref{starfree} is devoted to star-free graphs; this section contains our results for line graphs and for the minimum maximum matching in $2K_2$-free graphs. Section \ref{Cfree} deals with graphs excluding cycles of some fixed size. We summarize our results and give complexity dichotomies in Section \ref{concl}.

\section{Definitions and notations}
Given a graph $G=(V,E)$, its diameter $diam(G)$ is the maximum length of a shortest path over every pair of vertices. For a vertex $v\in V$, $N(v)$ denotes its open neighborhood, i.e. the set of the neighbors of $v$, $N[v]=N(v)\cup\{v\}$ it closed neighborhood.
For a subset $S\subseteq V$, we let $G[S]$ denote the subgraph of $G$ {\it induced} by $S$, which has vertex set~$S$ and edge set $\{uv\in E\; |\; u,v\in S\}$. For a vertex $v\in V$, we write $G-v=G[V\setminus \{v\}]$ and for a subset $V'\subseteq V$ we write $G-V'=G[V\setminus V']$. A set $M\subseteq E$ is a {\it matching} when no two edges share an endpoint. A matching $M$ is {\it maximal} when there is no edge $e\in E\setminus M$ such that $M\cup\{e\}$ is a matching. A \textit{minimum maximal matching} is a maximal matching such that $\vert M\vert$ is minimum.
A set $S\subseteq V$ is called a {\it stable set} or an {\it independent set} if any pairwise distinct vertices $u,v\in S$ are non adjacent. The maximum cardinality of an independent set in $G$ is denoted by $\alpha(G)$. A stable set $S$ is {\it maximal} if  every vertex outside $S$ has a neighbor in $S$. A set $S\subseteq V$ is called a {\it clique} if any pairwise distinct vertices $u,v\in S$ are adjacent. A vertex $v\in V$ is {\it simplicial} when $N(v)$ is a clique.  If $V$ is a clique, then $G$ is a {\it complete graph}. We denote by $K_p,p\ge 1,$ the clique or the complete graph on $p$ vertices; $k.K_p$ is the disjoint union of $k$ cliques. The bipartite clique $K_{p,q}$ is the graph where the vertex set is partioned into two stables sets $A$ and $B$ with $\vert A\vert=p,\vert B\vert=q$ such that there is an edge $ab$ for every pair $a\in A,b\in B$. The \textit{star} is $K_{1,p}$ and the {\it claw} is $K_{1,3}$. The induced cycle on $p$ vertices is denoted by $C_p$. When $p\ge 4$ then $C_p$ is a {\it hole}. The {\it girth} of $G$ denoted by $g(G)$ is the minimum length of a (induced) cycle.  For a fixed graph $H$, we write $H\subseteq_i G$ whenever  $G$ contains $H$ as an induced subgraph, and we say that $G$ is {\it $H$-free} if $G$ has no induced subgraph isomorphic to $H$. For a set of graphs ${\cal H}=\{H_1,\ldots,H_p\}$, $G$ is $(H_1,\ldots,H_p)$-free or $\cal H$-free if it has no induced subgraph isomorphic to $H_i,1\le i\le p$.

The problems we study throughout the paper concern graphs $G=(V,E)$ with fixed diameter $diam(G)=k$. Note that, a graphs with $diam(G)=1$ is a clique,  a graphs with $diam(G)=2$ is such that for any pair of vertices $u,v$ such that $uv\not\in E$ then there exists a vertex $w\ne u,v$ with  $w\in N(u)\cap N(v)$.

\section{$\boldmath{K_{\boldmath{1,l}}}$-free graphs}\label{starfree}

In this section we are interested in how the diameter and the size of a fixed forbidden star impact the computational complexity of \textsc{Dominating Set}. It is known from Yannakakis and Gavril \cite{Yannakakis} that \textsc{Dominating Set} is $\mathsf{NP}$-complete for claw-free graphs. Therefore we focus first on the following question: for which value of $d\geq 2$ does \textsc{Dominating Set} remains $\mathsf{NP}$-complete for claw-free graphs with fixed diameter $d$ ? For an initial response, we show that \textsc{Dominating Set} is $\mathsf{NP}$-complete for claw-free graphs with diameter $d\geq 3$. Then, it remains the case of claw-free graphs with diameter $2$. For that, we focus first on line graphs which is a subclass of claw-free graphs. For line graphs $L(G)$, we know that the complexity of \textsc{Dominating Set} is tightly related to the complexity of computing a minimum maximal matching in $G$. Moreover, Martin et al. \cite{Martin} have shown that if $L(G)$ is a line graph with diameter $2$, then $G$ is $2K_2$-free. Therefore we show that we can compute a minimum maximal matching in polynomial-time in $2K_2$-free graphs. Then we use known results of the literature to show that \textsc{Dominating Set} is polynomial-time solvable for the other claw-free graphs with diameter $2$. After delimiting all possible cases for the diameter and the claw, we answer the following question: for which stars does \textsc{Dominating Set} remains polynomial-time solvable for star-free graphs with diameter $2$ ? As a response, we show that \textsc{Dominating Set} is $\mathsf{NP}$-complete for $K_{1,4}$-free graphs with diameter $2$.

\subsection{Claw-free graphs with diameter $\boldsymbol{d\ge 3}$}
The main part of the proof of the following result deals with the claw-free graphs with diameter 3. We show how to adapt our construction for claw-free graphs with greater diameter at the end of the proof.
\begin{lemma}\label{kge3}
    For any integer $d\ge 3$, \textsc{Dominating Set} is $\mathsf{NP}$-complete for claw-free graphs with diameter $d$.
\end{lemma}
\begin{proof}
    We start with the case $d=3$. We give a polynomial transformation from \textsc{Dominating Set} which is $\mathsf{NP}$-complete for cubic graphs \cite{CubDom}. Let $G=(V,E)$ be a cubic graph and let $n$ be its number of vertices. From $G$ we build a claw-free graph $G'=(V',E')$ with diameter $3$. We replace every vertex $v$ by the subgraph $G_v$ as depicted by Figure \ref{gadgetv}. All the $3n$ vertices inside the three triangles of each $G_v$ (the grey vertices in Figure \ref{gadgetv}) are connected to a vertex $s$, and made pairwise connected to form a clique $K$ on $3n+1$ vertices. Last, a vertex $t$ is connected with $s$. The Figure \ref{constr3} is a simplified representation of $G'$.

    We show that $G'$ is claw-free. From Figure \ref{constr3}, we can see that $s$ and $t$ are not at the center of an induced claw. For each vertex $u$ of the $3n$ grey vertices, their neighborhood can be partition into the clique $K\setminus \{u\}$ and the clique on three white vertices inside a subgraph $G_v$. So no vertices of $K$ are at the center of an induced claw. It remains the white vertices of each subgraph $G_v$. Since the six inner vertices of $G_v$ are not connected to vertices of $G'-G_v$, it follows that they are not at the center of an induced claw. Each of the three outer vertices of $G_v$ are connected to exactly one vertex in $G'-G_v$. Since their neighborhood inside $G_v$ is a clique, it follows that these vertices are not at the center of an induced claw. Therefore $G'$ is claw-free.

    We show that $G'$ has diameter $3$. For each subgraph $G_v$, each of its vertices is connected to a grey vertex of $K$. Since $s\in K$ the vertex $t$ is at distance two to any grey vertex, so at distance three to each white vertex of any $G_v$. Any of the $3n$ grey vertices are at distance at most $2$ from every white vertices of the subgraphs $G_v$ while for any pair of such white vertices the distance is at most three. It follows that $G'$ has diameter $3$.

    \begin{figure}[htbp]
        \begin{center}
            \includegraphics[width=10cm, height=5cm, keepaspectratio=true]{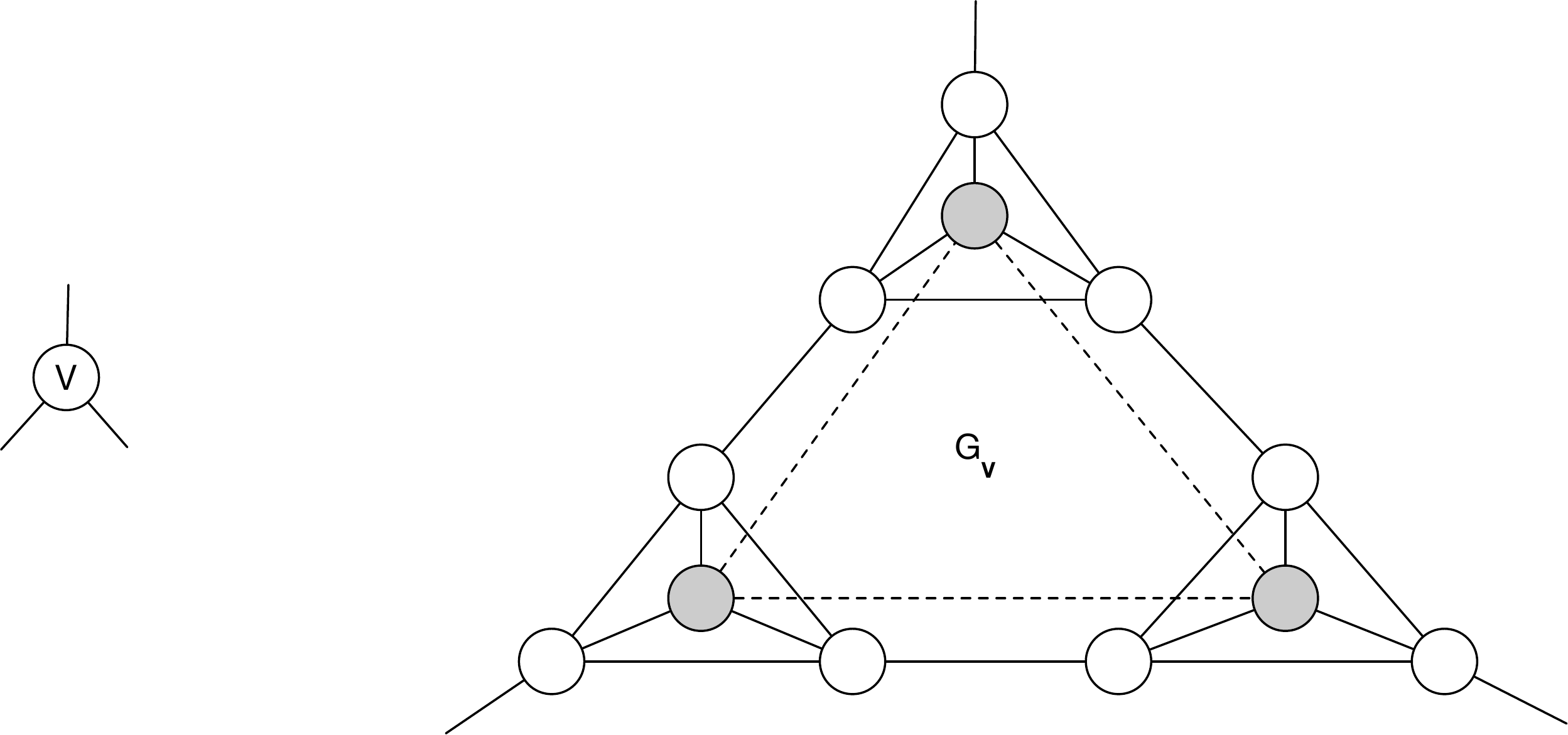}
        \end{center}
        \caption{The gadget associated with a vertex $v$. The grey vertices are in the clique $K$, and their corresponding edges are the dashed lines.}
        \label{gadgetv}
    \end{figure}

    \begin{figure}[htbp]
        \begin{center}
            \includegraphics[width=12cm, height=3cm, keepaspectratio=true]{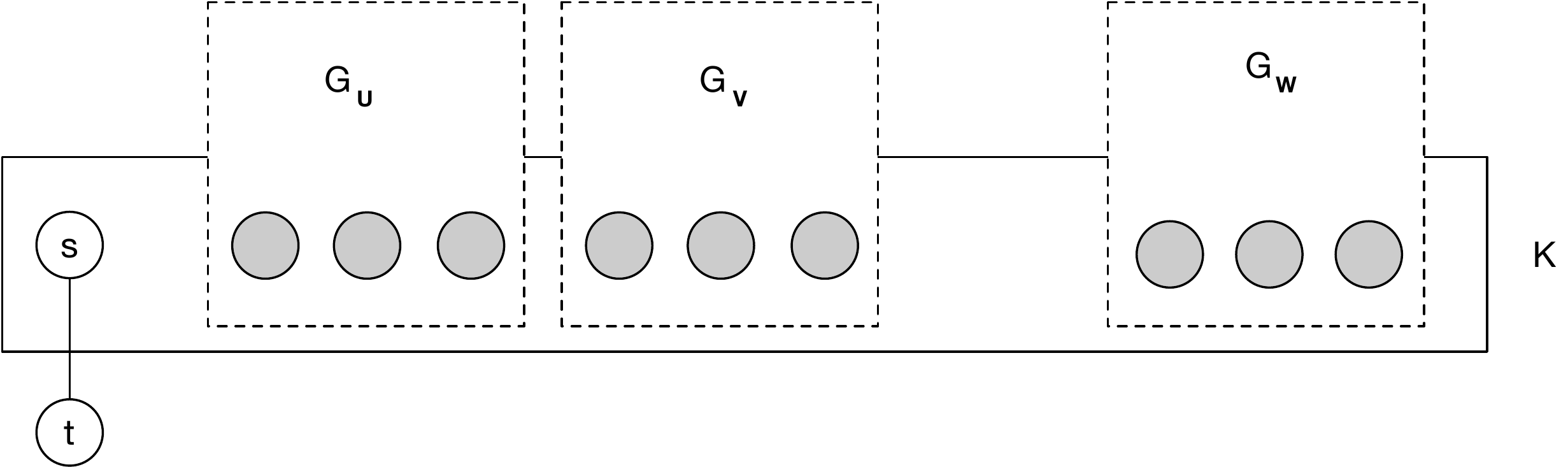}
        \end{center}
        \caption{An outline of $G'$. }
        \label{constr3}
    \end{figure}

    Let $k'=2n+k+1$. First, we show that if $G$ has a dominating set no greater than $k$, then $G'$ has a dominating set no greater than $k'$. From a dominating set $D$ of $G$, we construct a dominating set $D'$ of $G'$. Since $t$ is a leaf, we can add $s$ to $G'$. For each vertex $v\in D$, we add the three outer vertex of its associated subgraph $G_v$ to $D'$ as depicted on the left of Figure \ref{domk=3}. Note that these three vertices dominate $G_v$ and the three vertices of $N(V'(G_v))\setminus V'(G_v)$. Therefore it remains to dominate all the inner vertices of each $G_u$ of $G'$ associated with a vertex $u\in V\setminus D$. To do so, we add two inner vertices of $G_v$ as depicted on the right of Figure \ref{domk=3}. So if $G$ has a dominating set of size $k$, then $G'$ has a dominating set of size $k'$.

    Conversely, let $D'$ be a dominating set of size less or equal than $k'=2n+k+1$ of $G'$. Since $t$ is a leaf, we can suppose that $s\in D'$. For each subgraph $G_v$  its six inner vertices induce a $C_6$, so $\vert V'(G_v)\cap D'\vert \geq 2$. The three outer vertices dominate $G_v$ (see the left of Figure \ref{domk=3}). Hence $2\leq \vert V'(G_v)\cap D'\vert \leq 3$. Thus there are at most $k$ subgraphs $G_v$ such that $\vert V'(G_v)\cap D'\vert =3$ whereas for the others $\vert V'(G_v)\cap D'\vert =2$. When  $\vert V'(G_v)\cap D'\vert =2$  one of the three outer vertices is not dominated by a vertex of $V'(G_v)\cap D'$ (see the right of Figure \ref{domk=3}), so it is dominated by its neighbor in $G_u,u\ne v,$ where $uv$ is an edge of $G$. So taking $v\in D$ whenever $G_v$ contains three vertices of $D'$ and $v\not\in D$ otherwise, we obtain a dominating set of $G$ with size at most $k$. \\

    \begin{figure}[htbp]
        \begin{center}
            \includegraphics[width=12cm, height=6cm, keepaspectratio=true]{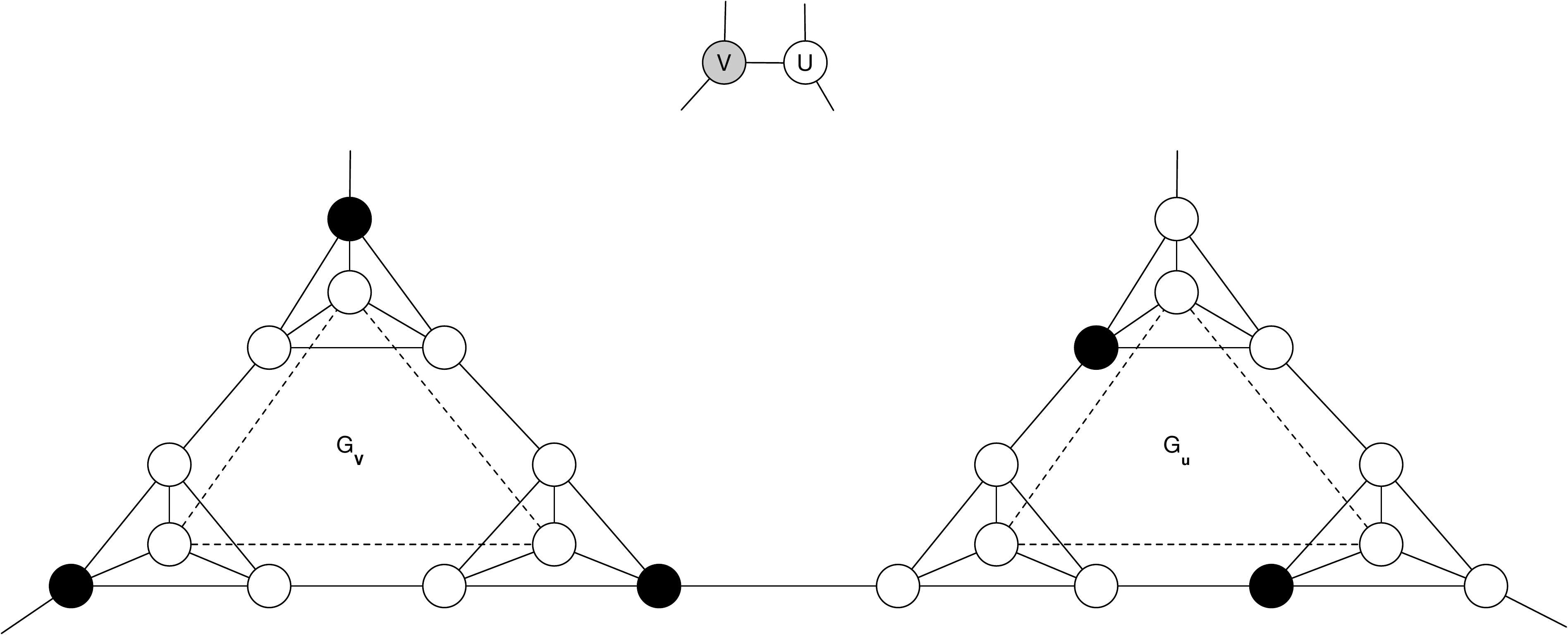}
        \end{center}
        \caption{$v$ is in a dominating set of $G$ and $u$ is not in the dominating set of $G$.}
        \label{domk=3}
    \end{figure}

    Now let $d\ge 4$. The construction of $G'$ is given by the figure \ref{constrd}. Taking $k'=2n+k+1+\lfloor {d\over 3}\rfloor$ the arguments are the same as above.
\end{proof}

    \begin{figure}[H]
        \begin{center}
            \includegraphics[width=12cm, height=3cm, keepaspectratio=true]{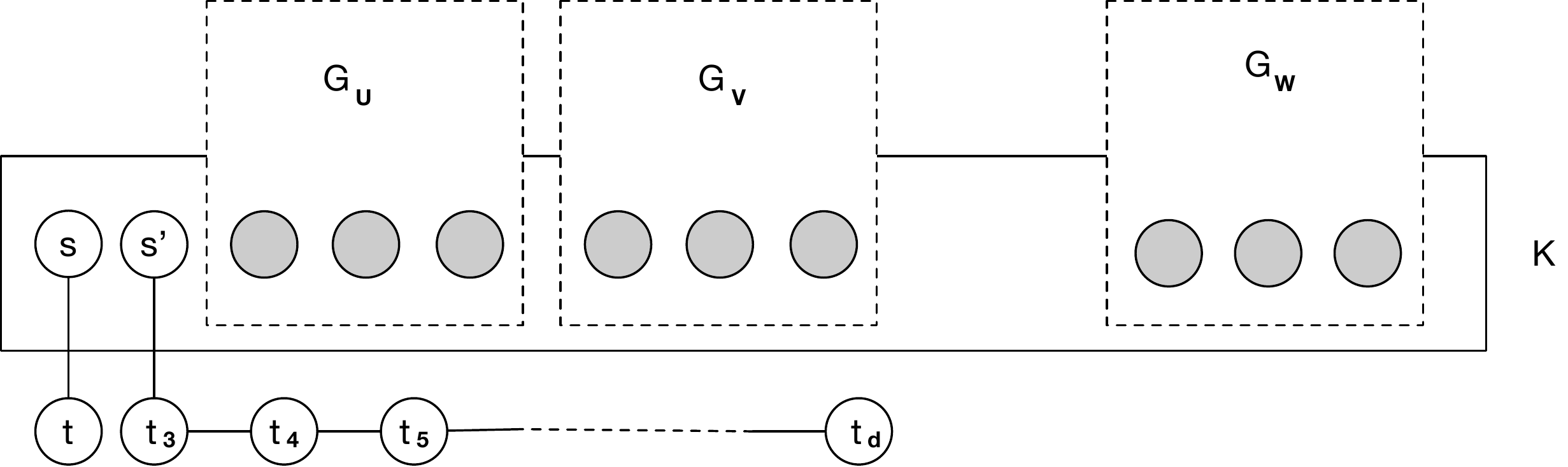}
        \end{center}
        \caption{An outline of $G'$ for $d\geq 4$. }
        \label{constrd}
    \end{figure}

\subsection{Line graphs with diameter $\mathbf 2$}

Given a graph $G$, the \textit{line graph} of $G$, denoted by $L(G)$, has for vertices the edges of $G$, and two vertices of $L(G)$ are adjacent if and only if the corresponding edges are incident in $G$. The class of line graphs can also be defined as the following list of forbidden subgraph:

\begin{theorem}[\cite{Brandstadt} p. 110]\label{linegraph}
    A graph $G$ is a line graph if and only if $G$ has no induced subgraph isomorphic to the nine forbidden graphs of Figure \ref{line_graphs_9forbidden_graphs}.
\end{theorem}

\begin{figure}[htbp]
    \begin{center}
        \includegraphics[width=12cm, height=7cm, keepaspectratio=true]{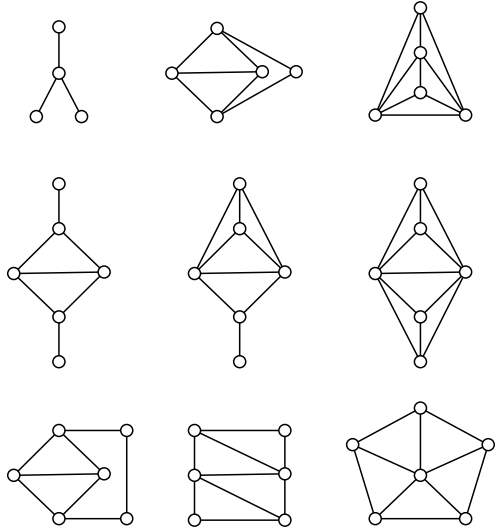}
    \end{center}
    \caption{ $\cal F$ the set of the nine forbidden subgraphs of line graphs.}
    \label{line_graphs_9forbidden_graphs}
\end{figure}

We denote by $\cal F$ the set of these nine forbidden graphs.
Since the line graphs can be defined by a finite set $\cal F$ of forbidden subgraphs, it follows that we can check if a graph is a line graph in polynomial-time. We give below the computational complexity of one of the best known algorithm to recognize line graphs.

\begin{prop}[Lehot \cite{Lehot}]\label{polylinegraph}
    We can test in $\vert E\vert +O(\vert V\vert)$ if a given graph $G=(V,E)$ is a line graph.
\end{prop}

Martin et al. \cite{Martin} (Lemma 11) showed that line graphs are closely related to $2P_2$-free graphs (i.e. $2K_2$-free graphs) with the following property:

\begin{prop}[Martin et al. \cite{Martin}]\label{diam2tok2}
    Let $G$ be a graph that is neither a triangle nor a star. Then $L(G)$ has diameter 2 if and only if $G$ is $2K_2$-free.
\end{prop}

From Yannakakis and Gavril \cite{Yannakakis} we know that a minimum maximal matching of $G$ corresponds to a minimum dominating set of $L(G)$.

Therefore, from Property \ref{diam2tok2}, we can deduce the complexity of \textsc{Dominating Set} for line graphs with diameter 2 from the complexity of computing a minimum maximal matching in $2K_2$-free graphs. It is know from Yannakakis and Gavril \cite{Yannakakis} that computing a minimum maximal matching is an $\textsf{NP}$-complete problem (even for cubic graphs). We show that we can compute a minimum maximal matching for $2K_2$-free graphs in polynomial time.



Before we start on that, we will need the following two results. The first result is the notourious theorem of Hall on matching of bipartite graphs.

\begin{theorem}[Hall's Theorem \cite{Bondy}]\label{Hall}
    A bipartite graph $G=(X\cup Y,E)$ has a matching that covers all vertices of $X$ if and only if

    \[\vert N(S)\vert \geq \vert S\vert \text{\ \ \ for every $S\subseteq X$}\]
\end{theorem}

The second result is of Farber \cite{Farber} who showed that $C_4$-free graphs have a polynomial number of cliques.

\begin{prop}[Farber \cite{Farber}]\label{poly_cliques_C4_free}
    If $G=(V,E)$ is a $C_4$-free graph, then $G$ has at most $\binom{\vert V\vert }{2}$ cliques.
\end{prop}

It follows that the graphs of its complement class, which is the class of $2K_2$-free graphs, have a polynomial number of maximal stable sets.

\begin{coro}\label{poly_stable_2K2}
    If $G=(V,E)$ is a $2K_2$-free graph, then $G$ has at most $\binom{\vert V\vert }{2}$ stable sets.
\end{coro}

In \cite{2K2maxind} Dhanalakshmi et al. used this last result to prove the following property.

\begin{prop}[Dhanalakshmi et al. \cite{2K2maxind}]\label{poly_maxstabl}
    If $G=(V,E)$ is a $2K_2$-free graph, then there exists an algorithm for enumerating all maximal stable sets, which runs in polynomial-time.
\end{prop}

We are ready to prove our result on minimum maximal matching for $2K_2$-free graphs. But to do so, we prove a more general result, that is:

\begin{lemma}\label{MMM2k2}
    Let $\mathcal{C}$ be the class of graphs with a polynomial number of maximal stable sets that can be enumerated in polynomial-time.

    Computing a minimum maximal matching for every $G=(V,E)\in \mathcal{C}$ can be done in polynomial-time.
\end{lemma}
\begin{proof}
    We give the structure of $G$ relatively to a minimum maximal matching $M_S$ of $G$.
    Let $Q_S$ be the vertices not covered by $M_S$. Since $M_S$ is maximal, it follows that $Q_S$ is either empty or a stable set. Therefore there is a maximal stable set $S$ of $G$ such that $Q_S\subseteq S$. Let $\mu'\subseteq M_S$ be the edges of $M_S$ with one endpoint in $S$, and $\mu = M_S\setminus \mu'$ be the remaining edges of $M_S$. Let $T(\mu)$ be the vertices of $V\setminus S$ not covered by $\mu$, that is, $T(\mu)=\{v\in N\setminus S \mid uv\in \mu'\}$.
    We claim that $T(\mu)$ is a stable set. For suppose not; then there is an edge $v_1v_2$ where $v_1,v_2 \in T(\mu)$. Since $u_1v_1,u_2v_2\in \mu'$, it follows that $M=(M_S\cup \{v_1v_2\})  \setminus \{u_1v_1,u_2v_2\}$ is a maximal matching with $\vert M\vert<\vert M_S\vert$, a contradiction. So $B_S(\mu)=G[T(\mu)\cup S]$ is bipartite. \medskip

    Therefore to each minimum maximal matching $M_S=\mu\cup \mu'$ corresponds $S$ a maximal stable set of $G$ such that $\mu$ is a maximal matching of $G[V\setminus S]$, and $\mu'$ is a maximum matching of $B_S(\mu)$ covering all vertices of $T(\mu)$. \medskip

    To build $M_S$ we do as follows. Hint $S$ a maximal stable set. Compute $\mu$ a maximum matching in $G[V\setminus S]$. Thus $T(\mu)$ -- the set of vertices of $V\setminus S$ not covered by $\mu$ -- is a stable set and $B_S(\mu)$ is a bipartite graph. Let $q(S)=\vert S\vert-\vert T(\mu)\vert$ and $\Theta(S)={\vert V\vert - q(S) \over 2}$. Note that from the structure of  a minimum maximal matching $M_S$ we have $\vert M_S\vert\le \Theta(S)$.

    If $B_S(\mu)$ has a maximum matching $\mu'$ covering all vertices $T(\mu)$, then $M=\mu\cup \mu'$ is a maximal matching of size $\Theta(S)=\vert \mu\vert+\vert \mu'\vert$. Henceforth, we say that such a $(S,\mu)$ is {\it fair}. \medskip

    We show that if $(S,\mu)$ is not fair, then there exists $({S_1},{\mu_1})$ with $\vert{S_1}\vert>\vert S\vert$ such that $\Theta({S_1}) < \Theta(S)$.
    Assume that $(S,\mu)$ is not fair. Then $B_S(\mu)$ has no matching that covers all vertices of $T(\mu)$. So from Thereom \ref{Hall}, there is $T'\subseteq T(\mu)$ such that $\vert N(T')\vert<\vert T'\vert$, where $N(T')=\{u\in N(v)\cap S \mid v\in T'\}$. Take $T'$ to be maximal.
    Let $P$ be a maximal stable set of the set of vertices $w$ covered by $\mu$ such that $N(w)\cap S\subseteq N(T')$ and $N(w)\cap T'=\emptyset$.
    Let $S_1=(S\setminus N(T'))\cup T'\cup P$. Then $S_1$ is a maximal stable set with
    \begin{align}
        \vert S_1\vert & = \vert S\vert+\vert T'\vert -\vert N(T')\vert+\vert P\vert \nonumber \\
                       & > \vert S\vert \nonumber
    \end{align}
    Let $\mu_1$ be a maximum matching in $G[V\setminus S_1]$. Note that $\vert T(\mu)\vert =\vert V\vert-\vert S\vert-2\vert \mu\vert $. Since $\vert \mu_1\vert \ge \vert\mu\vert-\vert P\vert$, it follows that
    \begin{align}
        \vert T(\mu_1)\vert & = \vert V\vert-\vert S_1\vert-2\vert \mu_1\vert \nonumber                  \\
                            & \leq \vert V\vert-\vert S_1\vert-2\vert \mu \vert +2\vert P\vert \nonumber
    \end{align}
    Therefore
    \begin{align}
        q(S_1)-q(S) & =\vert S_1\vert-\vert T(\mu_1)\vert -\vert S\vert+ \vert T(\mu)\vert \nonumber                                                                              \\
                    & \geq \vert S_1\vert-\vert V\vert+\vert S_1\vert+2\vert \mu\vert -2\vert P\vert-\vert S\vert+\vert V\vert-\vert S\vert-2\vert \mu\vert \nonumber \\
                    & \geq 2(\vert S_1\vert-\vert S\vert-\vert P\vert) \nonumber
    \end{align}
    Since $\vert S_1\vert-\vert S\vert-\vert P\vert=\vert T'\vert -\vert N(T')\vert>0$, it follows that $q(S_1)>q(S)$. Moreover, since $\Theta(S)={\vert V\vert - q(S) \over 2}$ and $\Theta(S_1)={\vert V\vert - q(S_1) \over 2}$, it follows that $\Theta(S_1)<\Theta(S)$. \medskip

    Our algorithm is as follows: for all the maximal stable sets $S$ we compute $\mu$ a maximum matching in $G[V\setminus S]$. If $(S,\mu)$ is fair, then we have a maximal matching of size $\Theta(S)$, else we do the same with the maximal stable set $S_1=(S\setminus N(T'))\cup T'\cup P$. Since $\vert S_1\vert>\vert S\vert$ the procedure stops with a maximal matching $S'$ of size  $\Theta(S')<\Theta(S)$. Taking a maximal matching of minimum size $\Theta(S)$ we have a minimum maximal matching $M_S$.

    Since $G$ has a polynomial number of maximal stable sets, and that we can enumerate them in polynomial-time, it follows that our algorithm is polynomial.\end{proof}


Since we can compute a minimum maximal matching for graphs with a polynomial number of maximal stable sets that can be enumerated in polynomial-time, it follows from Property \ref{diam2tok2}:

\begin{theorem}\label{mdsdiam2}
    \textsc{Dominating Set} is polynomial-time solvable for line graphs with diameter $2$.
\end{theorem}

\subsection{Claw-free graphs with diameter 2}
Since line graphs is a proper subclass of claw-free graphs, it remains the claw-free graphs that are not line graphs. We will use the following result of Martin et al. \cite{Martin} for claw-free graphs with diameter $2$.
\begin{theorem}[Martin et al. \cite{Martin}]\label{struct}
    Every claw-free graph $G$ with diameter $2$ with distinct neighbourhoods, no $W$-joins, $\alpha(G)>3$, and $\vert V (G)\vert> 13$ is a proper circular-arc graph or a line graph.
\end{theorem}

A graph $G$ has \textit{distinct neighbourhoods} if there is no pair of adjacent vertices $u,v$ such that $(N(u)\setminus \{v\})\subseteq (N(v)\setminus \{u\})$, i.e. $N[u]\subseteq N[v]$.

A $W$-\textit{join} is a pair $(A,B)$ of disjoint non-empty sets of vertices such that $\vert A\vert+\vert B\vert>2$, $A$ and $B$ are cliques, $A$ is neither complete nor anticomplete to $B$, and every vertex of $V (G) \setminus (A \cup B)$ is either complete or anticomplete to $A$ and either complete or anticomplete to $B$. Note that if $G$ is connected and contains a $W$-join, then for any pair $u\in A$, $v\in B$, $\{u,v\}$ is a $\gamma$-set of $G$ and so $\gamma(G)\le 2$.

A \textit{circular-arc} graph is the intersection graph of a set of arcs on the circle. A circular-arc graph is \textit{proper} if there is a corresponding arc model such that no arc contains another.

To prove that we can solve \textsc{Dominating Set} in polynomial-time for claw-free graphs with diameter $2$, we use the following results of Hsu and Tsai \cite{Hsu} on circular-arc graphs.

\begin{theorem}[Hsu and Tsai \cite{Hsu}]\label{dompolycircular}
    \textsc{Dominating Set} can be solve in $O(n)$ time for circular-arc graphs.
\end{theorem}


We are ready to prove the main result of this section.
\begin{theorem}\label{clawfree}
    \textsc{Dominating Set} is polynomial-time solvable for claw-free graphs with diameter $2$.
\end{theorem}
\begin{proof}
    Let $G=(V,E)$ be a claw-free graph with diameter $2$.
    We can assume that $\gamma(G)\geq 4$,  otherwise using brute force  we can compute a minimum dominating set of $G$ in time $O(n^5)$. Therefore $G$ has no $W$-join. Since $\gamma(G)\leq \alpha(G)$, it follows that $\alpha(G) > 3$. Moreover we can assume that $\vert V\vert > 13$ otherwise we can compute a $\gamma$-set in $O(1)$.

    Assume that there exists a pair of adjacent vertices $u,v$ such that $(N(u)\setminus \{v\})\subseteq (N(v)\setminus \{u\})$. For every $\gamma$-set $D$ of $G-u$, we have $N(u)\cap D \neq \emptyset$ in $G$. So a $\gamma$-set of $G-u$ is a $\gamma$-set of $G$. Moreover $diam(G-u)$=$diam(G)=2$. Therefore we can search for all such pairs $u,v$ and remove $u$ from $G$. This can be done in $O(n^3)$ time.

    Thus we can assume that $G$ has distinct neighbourhoods, no $W$-joins, $\alpha(G)>3$, and $\vert V (G)\vert> 13$. Then from Theorem \ref{struct}, $G$ is a proper circular-arc graph or a line graph. From Property \ref{polylinegraph}, we can check in polynomial time if $G$ is a line graph. If $G$ is a line graph, then from Theorem \ref{mdsdiam2} we can compute a $\gamma$-set of $G$ in polynomial time. Otherwise $G$ is a a proper circular-arc graph and from Theorem \ref{dompolycircular} we can compute a $\gamma$-set of $G$ in linear time.
\end{proof}

\subsection{$K_{1,4}$-free graphs with diameter 2}\label{K14}
In the section above we proved that \textsc{Dominating Set} is polynomial-time solvable for $K_{1,3}$-free graphs with diameter $2$. Here we show that the situation is different when we are concerned with $K_{1,4}$-free graphs.
\begin{theorem}\label{K14free}
\textsc{Dominating Set} is $\mathsf{NP}$-complete for $K_{1,4}$-free graphs with diameter $2$.
\end{theorem}
\begin{proof}
We give a polynomial transformation from \textsc{Vertex Cover} (see \cite{GJ}).
From $I=(G=(V,E),k)$ an instance of \textsc{Vertex Cover}, we build an instance $I'=(G',\gamma)$ where $G'$ is $K_{1,4}$-free with diameter $2$ and $\gamma=k$.

We start by constructing $G'=(V',E')$. The vertices of $V'$ are partitioned into $V_1,E_1,E_2,S,\{s\}$. We defined these sets and the edges of $G'$ as follows:

\begin{itemize}
    \item For each vertex $v\in V$, there is a vertex $v_1\in V_1$, that is, $V_1=\{v_1 \mid v\in V\}$;
    \item For each edge $uv\in E$, there is a vertex $e_{uv}^1\in E_1$ and $e_{uv}^2\in E_2$, that is, $E_1=\{e_{uv}^1 \mid uv\in E\}$ and $E_2=\{e_{uv}^2 \mid uv\in E\}$;
    \item For each $u_1\in V_1$, the vertices $\{u_1\}\cup \{e_{uv}^1 \mid u=u_1\}$ and $\{u_1\}\cup \{e_{uv}^2 \mid u=u_1\}$ form two cliques;
    \item $V_1\cup S\cup\{s\}$ is a clique;
    \item For each pair $e,e'\in E_1\cup E_2$ such that $N(e)\cap N(e')\cap V_1=\emptyset$, there are a vertex $s_{e,e'}\in S$ and the two edges $s_{e,e'}e,s_{e,e'}e'\in E'$. Note that these edges $e,e'$ correspond to copies of non-incident edges in $G$.
\end{itemize}

Since $V_1\cup S\cup\{s\}$ is a clique and that every pair of vertices $e,e'\in E_1\cup E_2$ has a common neighbor in $V_1\cup S$, it follows that $diam(G')=2$.
We show that $G'$ is $K_{1,4}$-free. For each vertex of $G'$, we give a partition of its neighborhood into at most three cliques. For $u_1\in V_1$: $N(u_1)\cap E_1$, $N(u_1)\cap E_2$, and $N(u_1)\cap (S\cup V_1\cup \{s\})$. For $e_{uv}^i\in E_i$, $i\in\{1,2\}$: $N(e_{uv}^i)\cap N(u_1)$, $N(e_{uv}^i)\cap N(v_1)$, and $N(e_{uv})\cap S$. For $s_{e,e'}\in S$: $\{s_{e,e'},e\}$, $\{s_{e,e'},e'\}$, and $V_1\cup S\cup \{s\}$. For the vertex $s$: $N(s)=S\cup V_1$. Therefore $G'$ is $K_{1,4}$-free. \medskip

Let  $C$, $\vert C\vert\le k=\gamma,$ be a vertex cover of $G$. Then its copy in $V_1$ is a dominating set of $G'$ of size at most $\gamma$.

Let $I'=(G',\gamma)$ be a positive instance, so there exists $\Gamma,\vert\Gamma\vert\le \gamma$ a dominating set of $G'$. From $\Gamma$ we will construct a dominating set $\Gamma'$ such that $\Gamma'\subseteq V_1$.
 Since $N(s)=V_1\cup S$ we can assume that  $s\not\in\Gamma$.
Let $S_i$ be the vertices of $S$ with two neighbors in $E_i$, that is, $S_i=\{s_{e,e'} \mid e,e'\in E_i\}$, $i=1,2$.
Let $\Gamma_i=\Gamma\cap (E_i,S_i)$. Without loss of generality $\vert \Gamma_1\vert \leq \vert \Gamma_2\vert$.
Let $\Gamma'=\Gamma\setminus \Gamma_2$. For each $e^1\in \Gamma^1$, we add $e^2$ to $\Gamma'$, and for each for each $s_{e^1,e'^1}\in \Gamma_1$, we add $s_{e^2,e'^2}$ to $\Gamma'$.
Since $G'[E_1\cup S_1\cup V_1]$ is isomorphic to $G'[E_2\cup S_2\cup V_1]$, it follows that $\Gamma'$ is a dominating set of $G'$ such that $\vert \Gamma'\vert \leq \gamma$.

Let $E_i'$ be the vertices $e_{uv}^i\in E_i$ such that $\Gamma'\cap N(e_{uv}^i)\cap V_1=\emptyset$, $i=1,2$. For each $e_{uv}^1\in E_1'$, there is $e_{uv}^2\in E_2'$, and vice versa, because $N(e_{uv}^1)\cap V_1 = N(e_{uv}^2)\cap V_1$. Since $N(e_{uv}^1)\cap N(e_{uv}^2)\cap S=\{s_{e,e'}\}$, with $e=e_{uv}^1,e'=e_{uv}^2$, and that each vertex of $S$ has exactly two neighbors in $E_1\cup E_2$, it follows that $\vert E_1'\vert \leq \vert S\cap \Gamma'\vert$. Then we remove the vertices of $S$ from $\Gamma'$ and we replace them by $u_1\in V_1$ for each $e_{uv}^1\in E_1'$. It follows that $\Gamma'$ is a dominating set of $G'$ such that $\vert \Gamma'\vert \leq \gamma$. Note that $\Gamma'\subseteq V_1$.

Let $C$ be the copies of the vertices of $\Gamma'\cap V_1$ in $G$. Since each vertex $e_{uv}^1\in E_1$ has a neighbor in $\Gamma'\cap V_1$, it follows that $C$ is a vertex cover of $G$ such that $\vert C\vert \leq k$.
\end{proof}

\section{Graphs with forbidden cycle(s)}\label{Cfree}

The section above concerns the graphs of diameter two that are star-free. Stars are special case of acyclic graphs. In this section we deal with graphs that are $H$-free when $H$ is a cycle or a collection of cycle of fixed size. The next sections concern first $C_3$-free graphs and then $(C_3,C_4)$-free graphs.

\subsection{Triangle-free graphs with diameter 2}
We first give a Lemma that will be used in the proof of our main result about $C_3$-free graphs.

\begin{lemma}\label{simplicial}
    Let $G$ be a graph. If $u,v$ are two distinct vertices such that $N(u)\subseteq N(v)$ and $v$ is simplicial, then $\gamma(G)=\gamma(G-v)$.
\end{lemma}
\begin{proof}
    We remark that for any graph with a simplicial vertex $w$ there exists a  $\gamma$-set not containing $w$.
    So there exists $S$ a $\gamma$-set of $G$ such that $v\not\in S$. Note that $S$ is a dominating set of $G-v$. Since $N(u)\subseteq N(v)$, it follows that $u$ is simplicial and $v$ is not a neighbor of $u$. So there exists $S'$ a $\gamma$-set of $G-v$ such that $u\not\in S'$. Then $S'$ is a dominating set of $G$. It follows that $\vert S\vert=\vert S'\vert$, i.e., $\gamma(G)=\gamma(G-v)$.
\end{proof}

We show the main result  concerning triangle-free graphs.
\begin{theorem}\label{d2trifree}
    \textsc{Dominating Set} is $\mathsf{NP}$-complete for triangle-free graphs with diameter $2$.
\end{theorem}
\begin{proof}
    We give a polynomial transformation from \textsc{Dominating set} which is $\mathsf{NP}$-complete for split graphs with diameter $2$ (see \cite{Splitdiam2}). From $I=(G,k)$ an instance of \textsc{Dominating set}, we build an instance $I'=(G',k')$. In $I=(G,k)$, $G=(K\cup S, E)$ is a split graph with diameter $2$ where $K$ is a clique and $S$ is a stable set. Let $u,v\in S$: first, since the vertices of $S$ are simplicial, it follows from Lemma \ref{simplicial} that we can suppose that $N(u)\not\subseteq N(v)$ and $N(v)\not\subseteq N(u)$; second, since $diam(G)=2$ there exists $w\in K$ such that $u-w-v$ is a path of $G$.
    From $G$ we build $G'=(V',E')$ as follows. We take a copie $K_1$ of $K$ and two copies $S_1,S_2$ of $S$. For the sake of simplicity, for $v\in K$, its copy in $K_1$ is denoted by $v_1$, whereas for $v\in S$, its copy in $S_1, S_2$, is denoted by $v_1, v_2$, respectively. Then we add two vertices $t$ and $s$. For each pair $u\in K$, $v\in S$: if $uv\in E$, then we add the edge $u_1v_1$, otherwise $uv\not\in E$ and we add the edge $u_1v_2$. For every $v\in S$, we add the edge $v_1v_2$. Then we make $t$ complete to $K_1$ and $s$ complete to $S_2$. Last we add the edge $st$. Note that $\{t\},K_1,\{s\},S_1,S_2$ is a partition of $G'$ into stable sets. Finally, we take $k'=k+1$. \medskip

    We show that $G'$ is triangle-free.
    Since $N(t)=K_1\cup \{s\}$ and $N(s)=S_2\cup \{t\}$ are two stable sets, it follows that $t$ and $s$ cannot be in a triangle. Thus if a triangle exists, it has one vertex $u_1\in K_1$, one vertex in $S_1$, and one vertex in $S_2$. So this triangle contains the edge $v_1v_2$. But when $u_1v_1$ is an edge, $u_1v_2$ is not an edge,  and vice versa. So $G'$ is triangle-free.\medskip

    We show that $diam(G')=2$. We observe that $t$ and $s$ are at distance at most two from any vertex of the graph. So we can focus on the vertices of $K_1\cup S_1\cup S_2$. Since $t$ is complete to $K_1$, respectively $s$ is complete to $S_2$, for any pair $v_1,u_1\in K_1$, respectively $v_2,u_2\in S_2$, there exists the path $v_1-t-u_1$, respectively $v_2-s-u_2$.  Since $diam(G)=2$, for any pair $v_1,u_1\in S_1$ there exists $w_1\in K_1$ such that $v_1-w_1-u_1$ is a path of $G'$. Now let $u_1\in S_1,v_1\in K_1$, respectively $u_2\in S_2,v_1\in K_1$, be such that $uv\not \in E$, respectively $uv \in E$. Then $u_2v_1\in E'$, respectively $u_1v_1\in E$, so $u_1-u_2-v_1$, respectively $u_2-u_1-v_1$, is a path of $G'$. Now let $u_1\in S_1,v_2\in S_2,u\ne v$. From Lemma \ref{simplicial}, we can assume that there exists  $w\in N(u),w\not\in N(v)$. Therefore $u_1w_1,v_2w_1\in E'$ and $u_1-w_1-v_2$ is a path of $G'$.
 So $diam(G')=2$. \medskip

    Let $D$ be a dominating set of $G$ with $\vert D\vert\le k$. Let $D'$ be the set of the copies of the vertices of $D$ in  $K_1\cup S_1$. Then $D'\cup \{t\}$ is a dominating set of $G'$ and $D'\cup \{k\}\leq k+1=k'$.

    Conversely, let $D'$ be a dominating set of $G'$ with $\vert D\vert\le k'=k+1$. Since $N(s)= S_2\cup \{t\}$, it follows that $\vert D'\cap(\{t,s\}\cup S_2)\vert\ge 1$. First, let $S_2\cap D'=\emptyset$. So $\vert D'\cap\{t,s\}\vert\ge 1$. For each $v_1\in D'\cap S_1$, if any, let a unique $u_1\in N(v_1)\cap K_1$. Then let ${\bar D}=(D'\setminus\{v_1: v_1\in D'\cap S_1\}\cup\{u_1: u_1\in N(v_1),u_1\ {\rm unique},\ v_1\in D'\cap S_1\})\setminus\{t,s\}$. Second, $\vert S_2\cap D'\vert\ge 1$, $t\in D'$. For each $v_2\in D'\cap S_2$, let a unique $u_1\in N(v_1)\cap K_1$ ($v_1$ is the neighbor of $v_2$ in $S_1$). Then  let ${\bar D}=(D'\setminus\{v_2 : v_2\in D'\cap S_2\}\cup\{u_1: u_1\in N(v_1),u_1\ {\rm unique},\ v_2\in D'\cap S_2\})\setminus\{t\}$.
    Third, $\vert S_2\cap D'\vert\ge 1$, $t\not\in D'$. For each $v_1\in S_1$ which is not dominated by a vertex of $K_1$ or by itself, we have that $v_1$ is dominated by $v_2$ its neighbor in $S_2$. Let any $w\in N(v_1)\cap K_1$. Since $t\not \in D'$ we have that $w$ is dominated either by $u_1\in N(w)\cap S_1,u_1\ne v_1$ or by $u_2\in S_2,u_2\ne v_2$.

    In the first case we have replaced $u_1$ by $w$ in $D'$, in the second case we have replaced $u_2$ by $w$ in $D'$. Then we take ${\bar D}=D'\cap(K_1\cup S_1)$. In all cases we take $D$ the copies of the vertices of $\bar D$. We have that $D$ is a dominating set of $G$ with $\vert D\vert\le k$.
\end{proof}

We derive the $\mathsf{NP}$-completeness of \textsc{Dominating Set} for triangle-free graphs for some classes of $H$-free graphs where $C_3\subseteq_i H$.
From Theorem \ref{linegraph}, the line graphs can be defined by the nine forbidden induced subgraphs of $\cal F$,  see Figure \ref{line_graphs_9forbidden_graphs}. All of them but the claw contain a triangle. From Theorem \ref{d2trifree}, it follows:

\begin{coro}\label{diam2linesclaw}
 \textsc{Dominating Set} is $\mathsf{NP}$-complete for $H$-free graphs with diameter $2$, $H\in{\cal F}\setminus \{claw\}$.
\end{coro}

\subsection{Graphs with fixed girth}
While, from  Theorem \ref{d2trifree}, \textsc{Dominating Set} is $\mathsf{NP}$-complete for $C_3$-free graphs with diameter $2$ , we show that the computational complexity status changes when $C_4$ is also a forbidden induced subgraph, that is, for the class of $(C_3,C_4)$-free graphs or, with other words, the graphs of girth  $5$ or more with diameter $2$. First we show that there is no need to study graphs of girth at least $6$ with diameter $2$ because there are none.

\begin{prop}\label{g>5}
    There is no graph $G$ with $diam(G)=2$ and girth $g(G)\ge 6$.
\end{prop}
\begin{proof}
    Assume there exists  $G=(V,E)$ with $g(G)\ge 6$ and $diam(G)=2$. Let $C=(v_1,v_2,\ldots,v_g,v_1)$, $g\ge6$, be a shortest hole in $G$. Since $diam(G)=2$  there exists $w\in N(v_1)\cap N(v_4)$, $w\not\in C$. But then $G[\{v_1,v_2,v_3,v_4,w\}]=C_5$, a contradiction since $g(G)\ge 6$.
\end{proof}

We give a property of graphs of girth $5$ with diameter $2$.

\begin{prop}\label{propgirth5}
   If $G=(V,E)$ is a graph with $g(G)=5,diam(G)=2$ then for any pair $u,v\in V,uv\not\in E$ we have $\vert N[u]\cap N[v]\vert = 1$.
\end{prop}
\begin{proof}
    Since $diam(G)=2$ we have  $N[u]\cap N[v]\ne\emptyset$.
    If $\vert N[u]\cap N[v]\vert \ge 2$, then $C_4 \subseteq_i G$.
\end{proof}

We are ready to show our result on graphs with diameter two and girth five.

\begin{prop}\label{girth5}
    Let $G=(V,E)$ be a graph of girth $5$ and diameter $2$. Then $\gamma(G)=\Delta(G)$ and we can find a minimum dominating set of $G$ in time $O(1)$.
\end{prop}
\begin{proof}
Since $C_5\subseteq_i G$ we have $\Delta(G)\geq 2$. The case where $\Delta(G)= 2$ corresponds to $G=C_5$ and the property is satisfied. So we can consider $\Delta(G)\geq 3$.  Let $u\in V$ such that $d(u)=\Delta(G)=\Delta$.

    Let $N(u)=\{v_1,\ldots,v_{\Delta}\}$. Let $W=V\setminus N[u]$ and let $W_i=\{w\in W \mid v_iw\in E\}$. Note that $\vert W_i \vert \leq \Delta-1$ and since $diam(G)=2$ we have $\bigcup W_i=W$. Since $G$ has no triangle each $W_i$ is an independent set. From Property \ref{propgirth5} there is no vertex $w\in W$ such that $w\in W_i\cap W_j$, $1\leq i < j\leq \Delta$. Therefore $W_1,\ldots,W_{\Delta}$ is a partition of $W$ into independent sets.

    Let $w\in W_i$. From Property \ref{propgirth5}, for every $j$ with $j\neq i$, $\vert N[w]\cap N[v_j]\vert =1$. It follows that $w$ has a neighbor in each $W_j$. Let $w,w'\in W_i$. From Property \ref{propgirth5} we have $N(w)\cap N(w') \cap W_j=\emptyset,i\ne j$. Therefore $\vert W_1\vert = \cdots = \vert W_{\Delta}\vert =\Delta-1$ and thus $G$ is $\Delta$-regular.

    Suppose that there exists  $S$ a minimum dominating set of $G$ such that $\vert S\vert < \Delta$. Then there exists an $i\in \{1,\ldots, \Delta\}$ such that $S\cap (W_i\cup \{v_i\})=\emptyset$. Since $N(v_i)\cap S\neq \emptyset$, it follows that $u\in S$. Since $\vert W_i\vert = \Delta-1$ and that for each pair $w,w'\in W_i$, $N[w]\cap N[w']=\{v_i\}$, it follows that $\vert S\cap W\vert = \Delta-1$. But then $\vert S\vert = \Delta$, a contradiction.

    Note that $N(u)$ is a dominating set of $G$ and thus $\gamma(G)=\Delta(G)$. Since $G$ is $\Delta$-regular, it follows that for any $x$, $N(x)$ is a minimum dominating set of $G$. Therefore we can compute a minimum dominating set in $O(1)$ time.
\end{proof}

We are ready to show a complexity dichotomy.
\begin{theorem}
    For the graphs of girth $g$ with diameter $2$, \textsc{Dominating Set} is $\mathsf{NP}$-complete if $g\in \{3,4\}$; and polynomial-time solvable if $g\geq 5$.
\end{theorem}
\begin{proof}
   From Property \ref{propgirth5} we know that  when a graph  has diameter $2$ its girth is no more than five. So with Property \ref{girth5} Dominating Set  is polynomial for $g\ge 5$.
    The graphs used in the proof of Theorem \ref{K14free} have girth $3$, those used  in the proof of Theorem \ref{d2trifree} have a girth $4$.
\end{proof}

\section{Conclusion}\label{concl}
Collecting the results of Sections \ref{starfree} and \ref{Cfree}  we give two (almost) dichotomies and we prove some partial results for the remaining case.

We give a dichotomy for claw-free graphs.
\begin{theorem}\label{dicot}
    For claw-free graphs with fixed diameter $d$, \textsc{Dominating Set} is $\mathsf{NP}$-complete if $d\geq 3$, it is polynomial-time solvable otherwise.
\end{theorem}
\begin{proof}
    Lemma \ref{kge3} proves the case $d\ge 3$.
    Theorem \ref{mdsdiam2} proves the case $d=2$.
    When $d=1$, the graph is a clique, so taking any vertex we obtain a minimum dominating set.
\end{proof}

We give our results for $H$-free graphs with diameter two.

\begin{theorem}\label{resdiam2}
    For $H$-free graphs with diameter $2$, \textsc{Dominating Set} is $\mathsf{NP}$-complete if $H'\subseteq_i H,H'\in \{2K_2,C_3,C_4,C_5,K_{1,4}\}$; it is polynomial-time solvable if  $H\subseteq_i H',H'\in \{P_4,K_{1,3}\}$.
\end{theorem}
\begin{proof}
    We know that the \textsc{Dominating Set} is $\mathsf{NP}$-complete for split graphs that are $(2K_2,C_4,C_5)$-free graphs. From Theorem \ref{d2trifree} and Theorem \ref{K14free} it is also $\mathsf{NP}$-complete for $C_3$-free and $K_{1,4}$-free graphs.

    \textsc{Dominating Set} is polynomial-time solvable for cographs that are $P_4$-free graphs because for such graphs $G$ with diameter $2$ we have $\gamma(G)\le 2$.   From Theorem \ref{clawfree} we have the result for claw-free graphs. \end{proof}

The connected graphs $H$ that are not involved in Theorem \ref{resdiam2} are the {\it chair}, that is, the graph with five vertices $v_i,1\le i\le 5,$ and four edges $v_1v_2,v_2v_3,v_1v_4,v_1v_5$, and the {\it biclaw}, that is, the graph with six vertices $v_i,1\le i\le 6,$ and five edges $v_1v_2,v_1v_3,v_1v_4,v_2v_5,v_2v_6$. The disconnected graphs are $\{K_{1,3}+aK_1,P_4+aK_1,chair+aK_1, biclaw+aK_1\},a\ge 1$.

We use Theorem \ref{resdiam2} to solve the cases for $(K_{1,3}+aK_1)$-free graphs and $(P_4+aK_1)$-free graphs.
\begin{coro}
    Let $a\ge 1$ be a fixed integer and $H\in  \{K_{1,3}+aK_1,P_4+aK_1\}$. \textsc{Dominating Set} is polynomial-time solvable for  the classes of $H$-free graphs with diameter $2$.
\end{coro}
\begin{proof}
    Let $H'\in \{K_{1,3},P_4\}$ and let $G$ be a $H$-free graph. From Theorem \ref{resdiam2}, we can compute a minimum dominating set of $G$ in polynomial-time when $G$ is $H'$-free.

    So we assume that $G$ contains $H'$ as an induced subgraph. Let $G'=G-N[H']$. Notice that $\alpha(G')<a$. Since $\gamma(G')\leq \alpha(G')$, it follows that $\gamma(G')<a$. Then $\gamma(G)\leq \gamma(N[H']) + \gamma(G')\le 4 + a - 1$.
We can find $H'$ in $O(n^4)$. Then by brute force, we can compute a minimum dominating set of $G$ in $O(n^{a+5})$.
\end{proof}

We left open the complexity status of \textsc{Dominating Set} for the classes of $(chair+aK_1)$-free graphs and of $(biclaw+aK_1)$-free graphs for any fixed nonnegative $a$.

The cases of $chair$-free graphs or $biclaw$-free graphs would be of special interest. Another open question concerns the class of line graphs. From Theorem \ref{clawfree} we know that \textsc{Dominating Set} is polynomial  for the line graphs with diameter $2$. Now, given any fixed integer $d\ge 3$, can \textsc{Dominating Set} be solved in polynomial-time for line graphs with diameter $d$ ?

\end{document}